\documentclass[12pt,a4paper]{amsart} 
\usepackage{hyperref}
\usepackage{mathtools}
\usepackage{amssymb}
\usepackage{mathrsfs}
\usepackage{graphics}
\newtheorem{theorem}{Theorem}[section]
\newtheorem{corollary}[theorem]{Corollary}
\newtheorem{example}[theorem]{Example}

\theoremstyle{definition}

\newtheorem{Remark}[theorem]{Remark}
\numberwithin{equation}{section}

\makeatletter
\@namedef{subjclassname@2020}{%
\textup{2020} Mathematics
Subject Classification}
 \makeatother
\title[Matrix Solutions to the Diophantine equation $X^4+ Y^4=2Z^4$]{ Matrix Solutions to the Diophantine equation $X^4+ Y^4=2Z^4$}

\author{SWADHIN MOHARANA$^{1}$}
\address{Department of Mathematics \\
   Berhampur University\\
   Bhanja Bihar 760007\\
   Ganjam, Odisha\\ India}
\email{swadhinmmath@gmail.com}

 \author{PABITRA KUMAR JENA$^{2}{^*}$}
\address{Department of Mathematics \\
   Berhampur University,
   Bhanja Bihar 760007,
   Ganjam, Odisha, India}
\email{pabitramath@gmail.com} 

\date{}
\thanks{*Corresponding author: Pabitra Kumar Jena, Email: pabitramath@gmail.com \\ The author P.K. Jena receives extra-mural research backing from the Government of Odisha via OSHEC, specifically under the MRIP-2024-Mathematics program (24EM/MT/85).}
\begin{document}

\begin{abstract}
The present study explores integer and rational  matrix solutions for the Diophantine equation $X^{4}+Y^{4}=2Z^{4}$, establishing that infinitely many solutions exist over different matrix rings, including $M_2(\mathbb{Z})$, $M_3(\mathbb{Z})$, $M_4(\mathbb{Z}).$  Furthermore, expanding the scope beyond the integer domain, the analysis demonstrates that an infinite number of matrix solutions can be found in degree-4 and degree-16 field extensions, specifically within $M_4(\mathbb{Q}(\sqrt[4]{l}))$ and $M_4(\mathbb{Q}(\sqrt[4]{l_1}, \sqrt[4]{l_2}))$ when the rational numbers involved lack perfect fourth powers and are not perfect squares.
\end{abstract}

\maketitle{}

\section{Introduction and Motivations}
Diophantine equations constitute one of the foundational and most profoundly challenging disciplines within the realm of number theory. Named in honor of the third-century mathematician Diophantus of Alexandria, who pioneered their systematic study in his seminal work \textit{Arithmetica}, a Diophantine equation is fundamentally a polynomial equation with integer coefficients where the explicit goal is to identify solutions restricted to the domain of integers or rational numbers \cite{hardy}. 

Diophantine problems are captivating because their seemingly simple formulations often require immensely complex mathematical strategies to solve. Furthermore, the discrete nature of integers prevents the use of continuous approximation techniques, ensuring that every equation presents an entirely unique challenge \cite{mordell}.

 Researchers traditionally categorize Diophantine equations by degree and variable count, utilizing specific techniques for each class. First-degree equations, like $ax + by = c$, depend on the Euclidean algorithm, possessing solutions only if $c$ is divisible by the greatest common divisor of $a$ and $b$ \cite{hardy}. Second-degree challenges, such as Pell's equation ($x^2 - ny^2 = 1$), always contain infinitely many integer roots discoverable via continued fractions \cite{mordell}. Conversely, higher-degree equations present extreme difficulties, best illustrated by Fermat's Last Theorem, which Andrew Wiles finally proved in 1995 using elliptic curves and Diophantine geometry \cite{wiles, hindry}. While Hilbert's Tenth Problem sought a universal mathematical algorithm to resolve any Diophantine equation, Yuri Matiyasevich proved in 1970 that such a finite, universal method is impossible \cite{matiyasevich}. Because this general undecidability prevents a singular computational approach, modern mathematical research primarily focuses on analyzing the underlying geometry of algebraic varieties to determine solvability \cite{hindry}.

 Although the classic scalar Diophantine equation $x^3+y^3=2z^3$ famously lacks any nontrivial integer solutions, its matrix equivalent—where the variables represent $n \times n$ matrices with integer entries—admits multiple nontrivial solution frameworks. By utilizing the algebraic substitutions $X=P+Q$, $Y=P-Q$, and $Z=P$, researchers can reduce the primary matrix equation into the simplified condition $Q^2P+QPQ+PQ^2=0$. Detailed is assured in  \cite{Garcia}.  Selecting specific companion matrices for $Q$ reveals that non-trivial integer configurations are entirely possible for the matrix formulation of this equation. 

\textbf{Important Mathematical Results} \cite{Garcia}:

\textbf{First Solution Family}: If the substitution variable $Q$ is explicitly defined as the companion matrix of the polynomial $x^2$, yielding $Q=\begin{bmatrix}0&1\\0&0\end{bmatrix}$, then the overarching matrix equation successfully evaluates to zero if and only if the matrix $P$ takes an upper triangular structure.

\textbf{Second Solution Family:}\cite{Garcia} 
When $Q$ is instead chosen to represent the companion matrix for the polynomial $x^2+x+1$, generating $Q=\begin{bmatrix}0&-1\\1&-1\end{bmatrix}$, a valid nontrivial integer matrix solution is guaranteed if and only if $P$ adopts the specific structural pattern $P=\begin{bmatrix}a&b\\a+b&-a\end{bmatrix}$.

Let $M_n(\mathbb Z)$ represents the ring of $ n \times n $ matrices over $\mathbb Z.$ Let us define  $\mathbb{Q}(\sqrt[4]{l})$ and $M_4(\mathbb{Q}(\sqrt[4]{l_1}, \sqrt[4]{l_2}))$ be the finite extension field over rational numbers $\mathbb Q.$ 
\vspace{1em}

In 2026, Mallick and Mishra \cite{Mallick1} investigated non-trivial integer matrix solutions for the Diophantine equation $X^3 + Y^3 = 2Z^3$, expanding upon previous findings that were limited to second-order matrices. By leveraging specialized structural properties of Circulant and Rare matrices—which are subclasses of Toeplitz matrices—the authors construct new infinite families of solutions for matrices of order 2, 3, and arbitrary dimensions $n$. Furthermore, the research broadens the scope of these solutions by applying them to algebraic number fields involving distinct cube-free rationals, demonstrating a deep interplay between matrix algebra and field extensions.

\textbf{Second-Order Matrix Solutions \cite{Mallick1}:} For a specific singular integer matrix $B = \begin{pmatrix} a & b \\ c & -a \end{pmatrix}$ where $a \neq 0$ and $a \neq \pm c$, the matrix relation transformed to $B^2A + BAB + AB^2 = O$ yields valid solutions if and only if the upper triangular companion matrix $A$ takes the explicitly parameterized form $A = \begin{pmatrix} w-ct & at \\ 0 & w \end{pmatrix}$ for arbitrary integers $t$ and $w$.

\textbf{Infinite Solutions in Subsets\cite{Mallick1}:} There exist infinitely many matrix solutions to the equation $X^3 + Y^3 = 2Z^3$ across specific matrix domains, notably within a custom closed set $\mathbb{H}$ of $2 \times 2$ matrices defined by a column-sum invariant ($a+c=b+d$), and within the broader class of third-order integer circulant matrices.

\textbf{Third-Order Solutions via Trace Conditions\cite{Mallick1}:} When $B$ is a strictly superdiagonal matrix of order 3, valid solutions for the associated matrix $A$ require that $A$ is a $3 \times 3$ integer matrix having a trace of zero, alongside specifically constrained lower-row entries reflecting $g=0$ and $h=-d$.

\textbf{Solutions Over Number Fields \cite{Mallick1}:} The existence of infinitely many matrix solutions for $X^3 + Y^3 = 2Z^3$ extends beyond standard integers into matrices evaluated over algebraic number fields of degree 3 and 9, specifically $M_3(\mathbb{Q}(\sqrt[3]{d}))$ and $M_3(\mathbb{Q}(\sqrt[3]{d_1}, \sqrt[3]{d_2}))$ for distinct, cube-free rational numbers $d, d_1,$ and $d_2$.

Researchers in  \cite{Mallick2} further explore the solvability of specific matrix Diophantine equations by establishing direct mathematical connections with their scalar counterparts over particular number fields. The authors demonstrate that finding a valid set of $2 \times 2$ integer matrices within a defined class $G_2(d)$ for the equation $X^n+Y^n=n!Z^n$ is algebraically equivalent to solving the corresponding scalar equation $\theta^n+\eta^n=n!\omega^n$ within the quadratic field $\mathbb{Q}(\sqrt{d})$. Furthermore, the study introduces a determinant-restricted matrix family, $G_2(d,l)$, to conclusively prove that the Fermat-like equation $X^n+Y^n=2Z^n$ possesses no non-trivial solutions when the exponent is strictly greater than two. 

\begin{theorem}\cite{Mallick2}
For any integer $n > 1$, the matrix Diophantine equation $X^n+Y^n=n!Z^n$ successfully yields a solution within the specialized $2 \times 2$ matrix set $G_2(d)$ if and only if there are specific scalar values $\theta, \eta, \omega \in \mathbb{Q}(\sqrt{d})$ that satisfy the algebraic relation $\theta^n+\eta^n=n!\omega^n$.    
\end{theorem}

\begin{theorem}\cite{Mallick2}
Given an arbitrary number field $K$ and a natural number $n \in \mathbb{N}$, if the scalar elements $p, q, r \in K$ solve the relation $p^{2n}+q^{2n}=n!r^{2n}$, then the specifically structured matrices $P$, $Q$, and $R$—constructed using those respective scalars in an anti-diagonal format—will intrinsically satisfy the elevated matrix equation $P^{4n}+Q^{4n}=n!R^{4n}$.   
\end{theorem}

\begin{theorem}\cite{Mallick2}
 Whenever the matrices $X, Y,$ and $Z$ are selected from the determinant-fixed subset $G_2(d,l)$, the Fermat-type matrix equation $X^n+Y^n=2Z^n$ is strictly restricted to only trivial solutions for any integer exponent $n > 2$.    
\end{theorem}

Drawing upon a broad base of theoretical  analysis on  Diophantine equations, as well as their generalizations \cite{Mallick1, Mallick2}, this article  presents the new expansions of integer matrix solutions for the Diophantine equation $X^{4}+Y^{4}=2Z^{4}$, establishing that infinitely many solutions exist over different matrix rings, including $M_2(\mathbb{Z})$, $M_3(\mathbb{Z})$, $M_4(\mathbb{Z})$ and $M_4(\mathbb{Q}(\sqrt[4]{l}))$ and $M_4(\mathbb{Q}(\sqrt[4]{l_1}, \sqrt[4]{l_2}))$.

\section{Integer matrix Solution of order $\texorpdfstring{2 \times 2}{2 x 2}$ to Diophantine equation}\label{section3}
The Diophantine equation 
\begin{equation}\label{e1}
 u^4+ v^4=2t^4  
\end{equation} enters into the picture when one arises fourth powers in arithmetic progression: if $u^4 < t^4 < v^4$  are in arithmetic progression, then $v^4- t^4= t^4- u^4$ which produces the equation \ref{e1}. This equation \ref{e1} has no integer solutions besides the trivial solution $\pm u=\pm v=\pm t.$ We think about matrix analogue $X^4+Y^4=Z^4$ of the equation  \ref{e1}.   
We are looking for solution to the diophantine equation\begin{equation}\label{a}
    X^4+Y^4=2Z^4
\end{equation}in the form \cite{Garcia} of X=U+V,Y=U-V and Z=U, then above equation \ref{a} reduces to \begin{equation}\label{b}
   U^2V^2+UVUV+UV^2U+VU^2V+VUVU+V^2U^2+V^4=0.
\end{equation}
\begin{theorem} \label{thm1}
 Let $0 \ne u\in \mathbb Z, u\ne \pm t \in \mathbb Z$ and  $V=\begin{pmatrix}
     u&v\\t&-u
 \end{pmatrix} \in M_2(\mathbb Z)$ with $det(V)=0$, then equation \ref{b} has solutions if and only if the matrix U is of the form    $U=\begin{pmatrix}
     x&y\\z&w
 \end{pmatrix}\in M_2(\mathbb Z)$ where $x=w=k,y=-nv,z=nt,$ and
 $n,k\in \mathbb Z.$

 \begin{proof}
  Since  $V=\begin{pmatrix}
     u&v\\t&-u
 \end{pmatrix}$ implies  $V^2=\begin{pmatrix}
     u^2+vt&0\\0&u^2+vt
 \end{pmatrix}=\begin{pmatrix}
     0&0\\0&0
 \end{pmatrix},$ then equation \ref{b} reduces to \begin{equation}\label{c}
     UVUV+VU^2V+VUVU=0
 \end{equation} which implies \begin{equation}
     (UV+VU)^2=0.
 \end{equation} So its trace and determinant are zero. Let's begin with $U=\begin{pmatrix}
     x&y\\z&w
 \end{pmatrix}\in M_2(\mathbb Z)$  that  computes $UV+VU=\begin{pmatrix}
     2ux+ty+vz&xv+vw\\tw+tx&vz+ty-2uw
 \end{pmatrix}.$ Now, from $trace(UV+VU)=0,$ we have $ux+ty+vz-uw=0.$\\

  If we choose $w=x$ we have $ty+vz=0 \implies y=-\frac{v}{t}z.$ In order to make $y,z$ integers, we must have to take $z=nt,n\in \mathbb Z$ so that $y=-nv.$ Hence,  $U=\begin{pmatrix}
     x&y\\z&w
 \end{pmatrix}\in M_2(\mathbb Z)$ where $w=x=k,z=nt,y=-nv,n,k\in \mathbb Z.$ \\
   Conversely, let $U=\begin{pmatrix}
     x&y\\z&w
 \end{pmatrix}\in M_2(\mathbb Z)$ where $w=x=k,z=nt,y=-nv,n\in \mathbb Z.$ then $UV+VU=\begin{pmatrix}
     2uk&2vk\\2tk&-2uk 
 \end{pmatrix}.$ So computing $(UV+VU)^2=\begin{pmatrix}
     4k^2(u^2+vt)&0\\0&4k^2(u^2+vt)
 \end{pmatrix}=\begin{pmatrix}
     0&0\\0&0
 \end{pmatrix}$ as $det(V)=0.$ Hence the result follows.
    
 \end{proof}
\end{theorem}

\begin{example}
    Let $V=\begin{pmatrix}
        2&-1\\-4&-2\\
    \end{pmatrix}$ then $U=\begin{pmatrix}
        3&-2\\-8&3
    \end{pmatrix}$ satisfies  the equation  \ref{b}.So X=U+V=$\begin{pmatrix}
        5&-1\\-12&1\\
    \end{pmatrix},$ Y=U-V=$\begin{pmatrix}
        1&-3\\-4&5
    \end{pmatrix},$ Z=U=$\begin{pmatrix}
        3&-2\\-8&3
    \end{pmatrix}$ is solution to equation  \ref{a}.
\end{example}

\begin{corollary}
 Let V be a matrix of the form $V=\begin{pmatrix}
     u&v\\t&-u
 \end{pmatrix} \in M_2(\mathbb Z)$ with $u\ne0,u\ne \pm t$ and $det(V)=0$,then equation \ref{b} has solutions if and only if the matrix U is of the form  $U=\begin{pmatrix}
     x&0\\0&w
 \end{pmatrix}\in M_2(\mathbb Z)$ where $x=w=k$ where $k\in \mathbb Z.$
\end{corollary}
\begin{proof}
 If we let y,z=0 in theorem \ref{thm1}, then we have x=w so $U=\begin{pmatrix}
     x&0\\0&w
 \end{pmatrix}\in M_2(\mathbb Z)$ where $x=w=k,k \in \mathbb Z.$    
\end{proof}

\begin{corollary}
   Let V be a matrix of the form $V=\begin{pmatrix}
     u&v\\t&-u
 \end{pmatrix} \in M_2(\mathbb Z)$ with $u\ne0,u\ne \pm t$ and $det(V)=0$,then equation \ref{b} has infinitely many solutions if and only if the matrix U is of the form $U=\begin{pmatrix}
     x&y\\0&w
 \end{pmatrix}\in M_2(\mathbb Z)$ where $x=t,w=-t,y=-2u.$  
\end{corollary}
\begin{proof}
    If we choose $w=-x, z=0$, then it follows from theorem \ref{thm1} that  $2ux+ty=0\implies ty=-2ux.$ For integer solutions we have to choose $x=t$ then $y=-2u.$ Hence  $U=\begin{pmatrix}
     x&y\\0&w
 \end{pmatrix}\in M_2(\mathbb Z)$ where $x=t,w=-t,y=-2u.$ Converse part can be proved easily. 
\end{proof}

\begin{Remark}
  If $V=\begin{pmatrix}
      0&k\\0&0
  \end{pmatrix}$ then the equation \ref{b} has solution iff U is of the form $\begin{pmatrix}
      u&v\\0&r
  \end{pmatrix},$ where $k,u,v,r\in \mathbb Z.$  
\end{Remark}
    
\begin{Remark}
    If $V=\begin{pmatrix}
      k&k\\k&k
  \end{pmatrix}$ then equation \ref{b} has solution iff U is of the form $\begin{pmatrix}
      u&-u\\-u&u
  \end{pmatrix},$ where $k,u\in \mathbb Z.$  
\end{Remark}

\begin{theorem}
    The  Diophantine equation $X^4+Y^4=2Z^4$ has infinitely many solutions in $M_2(\mathbb Z)$. 
\end{theorem}
\begin{proof}
    First we will introduce three matrix $R=\begin{pmatrix}
     1&1\\1&-1
 \end{pmatrix}$, $S=\begin{pmatrix}
     0&2\\1&0
 \end{pmatrix}$, $T=\begin{pmatrix}
     0&1\\2&0
 \end{pmatrix}$ which clearly satisfy equation $R^4+S^4=2T^4$. For $U=\begin{pmatrix}
     u&0\\0&u
 \end{pmatrix}\in M_2(\mathbb Z)$ we have $(UR)^4=U^4R^4,(US)^4=U^4S^4,(UT)^4=U^4T^4$. Now multiplying $U^4$ on both side of equation  $R^4+S^4=2T^4$ we get $U^4R^4+U^4S^4=2U^4T^4 \implies(UR)^4+(US)^4=2(UT)^4.$ Hence $X=\begin{pmatrix}
     u&u\\u&-u
 \end{pmatrix},Y=\begin{pmatrix}0&2u\\u&0
     \end{pmatrix},Z=\begin{pmatrix}
         0&u\\2u&0
     \end{pmatrix},u \in \mathbb Z$ satisfies the Diophantine equation.
     \end{proof}

\begin{corollary}
    The  Diophantine equation $X^n+Y^n=2Z^n$ has infinitely many solutions in $M_2(\mathbb Z)$ for each even n $\in \mathbb N.$     
\end{corollary}
\section{ matrix solution of order $3 \times 3$ }
This section contains integer matrix solution of order 3 to the Diophantine equation $X^4+Y^4=2Z^4.$
\begin{theorem}
    If $V=\begin{pmatrix}
        0&1&0\\0&0&1\\0&0&0
    \end{pmatrix}$ then $X^4+Y^4=2Z^4$ has solutions if and only if U is of the form $\begin{pmatrix}
        0&b&c\\0&0&f\\0&0&0
    \end{pmatrix} \in M_3(\mathbb Z) or\begin{pmatrix}
        0&k&s\\t&0&k\\0&-t&0
    \end{pmatrix} \in M_3(\mathbb Z).$
\end{theorem}
\begin{proof}
As $V=\begin{pmatrix}
    0&1&0\\0&0&1\\0&0&0
\end{pmatrix}\implies V^2=\begin{pmatrix}
    0&0&1\\0&0&0\\0&0&0
\end{pmatrix}\implies V^4=0.$ Now equation\ref{b} reduce to $U^2V^2+UVUV+UV^2U+VU^2V+VUVU+V^2U^2=0.$ Let $U=\begin{pmatrix}
    a&b&c\\d&e&f\\g&h&i
\end{pmatrix}.$ solving $U^2V^2+UVUV+UV^2U+VU^2V+VUVU+V^2U^2=0$ we get 
$2ag+dh+ig+d^2+eg=0,2ad+2ed+2eh+ah+2bg+fg+ih=0,a^2+e^2+i^2+2bd+bh+2fh+2cg+ae+ai+df+ei=0,2dg+hg=0,d^2+h^2+2dh+2eg+ag+ig=0,ad+bg+2gf+2eh+2ed+2ih+id=0,g^2=0,2gh+dg=0,h^2+ag+dh+2ig+eg=0.$ from $g^2=0,$ we have $g=0$ now using $g=0$ we also get $(d+h)^2=0\implies d=-h$.Again using $d=-h$ we get $(a-i)d=0$ \\
Case-1($d=0$)\\
 As $a^2+e^2+i^2+2bd+bh+2fh+2cg+ae+ia+df+ei=0\implies a=e=i=0.$ Hence $U=\begin{pmatrix}
    0&b&c\\0&0&f\\0&0&0\\
\end{pmatrix}\in M_3(\mathbb Z).
$Converse can be proved easily.\\
Case-2($d\neq0$)\\
From above equation we have $3a^2+e^2+2ae+h(f-b)=0\implies2a^2+(e+a)^2+h(f-b)=0$ that means $h$ must divides $2a^2+(e+a)^2.$ Setting $a=e=0 $ we get $f=b$ and c is arbitrary.
So we have $U=\begin{pmatrix}
    o&k&s\\t&0&k\\0&-t&0\\
\end{pmatrix}$ where $k,s,t\in \mathbb Z.$ Converse can be done easily.

\end{proof}

\begin{corollary}
     If $V=\begin{pmatrix}
        0&j&0\\0&0&j\\0&0&0
    \end{pmatrix}  \in M_3(\mathbb Z),j\neq 0$ then $X^4+Y^4=2Z^4$ has solutions if and only if U is of the form $\begin{pmatrix}
        0&b&c\\0&0&f\\0&0&0
    \end{pmatrix} \in M_3(\mathbb Z) or\begin{pmatrix}
        0&k&s\\t&0&k\\0&-t&0
    \end{pmatrix} \in M_3(\mathbb Z).$
\end{corollary}

\begin{theorem}
    The  Diophantine equation $X^4+Y^4=2Z^4$ has infinitely many solutions in $M_3(\mathbb Z)$. 
\end{theorem}

\begin{proof}
  Let us consider three matrix $P=\begin{pmatrix}
      0&1&0\\1&0&0\\0&0&1
  \end{pmatrix},Q=\begin{pmatrix}
      0&0&1\\0&1&0\\1&0&0
  \end{pmatrix},R=\begin{pmatrix}
      1&0&0\\0&0&1\\0&1&0
  \end{pmatrix}$ which  clearly satisfy $P^4+Q^4=2R^4.$ Now for matrix $L=\begin{pmatrix}
      t&0&0\\0&t&0\\0&0&t
  \end{pmatrix}\in M_3(\mathbb Z), $  we have $L^4P^4=(LP)^4,L^4Q^4=(LQ)^4,L^4R^4=(LR)^4.$ 
  On multiplying $L^4$ to equation $P^4+Q^4=2R^4$ ,we get $(LP)^4+(LQ)^4=2(LR)^4$.Hence $X=\begin{pmatrix}
      0&t&0\\t&0&0\\0&0&t
  \end{pmatrix},Y=\begin{pmatrix}
      0&0&t\\0&t&0\\t&0&0
  \end{pmatrix},Z=\begin{pmatrix}
      t&0&0\\0&0&t\\0&t&0
  \end{pmatrix}$ satisfy $X^4+Y^4=2Z^4,$ for $t\in \mathbb Z.$
\end{proof}

\section{Matrix solution of order $4 \times 4$}
In this section, we have found integer matrix solutions  as well as solutions  
over extension fields to the same Diophantine equation.
\begin{theorem}
 Let $V=\begin{pmatrix}
     0&0&k&0\\0&0&0&k\\0&0&0&0\\0&0&0&0
 \end{pmatrix}$ ,then equation \ref{b} has solution if $U$ is of the form $\begin{pmatrix}
     U_{11}&U_{12}\\0_2&U_{22}
 \end{pmatrix}$ where$U_{11},U_{12},U_{22},0_2 $ are block matrix of order 2.  
\end{theorem}

\begin{proof}
  We can write matrix $V=\begin{pmatrix}
    0_2&kI_2\\0_2&0_2  
  \end{pmatrix}$  as here $V^2=0 $ hence equation\ref{b} reduces to $UVUV+VU^2V+VUVU=0.$ Let $U=\begin{pmatrix}
     U_{11}&U_{12}\\U_{21}&U_{22}
 \end{pmatrix},$ now $UVUV+VU^2V+VUVU=\begin{pmatrix}
     U_{21}^2&U_{11}U_{21}+U_{21}U_{22}+U_{21}U_{11}+U_{22}U_{21}\\
     0&U_{21}^2
 \end{pmatrix}$ The above expression became zero if $U_{21}=0.$ Hence if U is of the form $\begin{pmatrix}
     U_{11}&U_{12}\\0_2&U_{22}
 \end{pmatrix}.$  
\end{proof}

\begin{theorem}
 let $V=\begin{pmatrix}
     1&1&1&1\\1&1&1&1\\1&1&1&1\\1&1&1&1
 \end{pmatrix},$ then equation \ref{b} has solution iff $U$ has row  and column sum zero.   
\end{theorem}

\begin{proof}
  We have $V=\begin{pmatrix}
     1&1&1&1\\1&1&1&1\\1&1&1&1\\1&1&1&1
 \end{pmatrix}.$ For any matrix U,  $(UV)_{i,j}$=row sum of $ith$ row of $U$ and $(VU)_{i,j}$= column sum of $jth$ column of $U,$ for $i,j=1, \cdots, 4.$ We can rewrite the equation\ref{b} as $UUVV+UVUV+VUVU+UVVU+VUUV+VVU=0$ if we observe we see that each term contain either UV or VU ,if we find such U for which both UV and VU is zero, equation will automatically satified hence equation will hold for those matrix U whose row sum and column sum is zero. Converse can be proved easily. 
\end{proof}

\begin{example} Let
     $V=\begin{pmatrix}
     1&1&1&1\\1&1&1&1\\1&1&1&1\\1&1&1&1
 \end{pmatrix}$ , $U=\begin{pmatrix} 1 & 2 & -1 & -2 \\ -2 & 1 & 2 & -1 \\ -1 & -2 & 1 & 2 \\ 2 & -1 & -2 & 1  \end{pmatrix} \in M_4 (\mathbb Z)$ satisfies the equation \ref{b}.
\end{example}

\begin{theorem}
      The  Diophantine equation $X^4+Y^4=2Z^4$ has infinitely many solutions in $M_4(\mathbb Z)$.
\end{theorem}

\begin{proof}
 Let us consider three matrix $P=\begin{pmatrix}
      0&1&0&0\\1&0&0&0\\0&0&0&1\\0&0&1&0
  \end{pmatrix},Q=\begin{pmatrix}
      0&0&1&0\\0&0&0&1\\1&0&0&0\\0&1&0&0
  \end{pmatrix},R=\begin{pmatrix}
      0&0&0&1\\0&0&1&0\\0&1&0&0\\1&0&0&0
  \end{pmatrix}$ which  clearly satisfy $P^4+Q^4=2R^4.$ Now for matrix $L=\begin{pmatrix}
      t&0&0&0\\0&t&0&0\\0&0&t&0\\0&0&0&t
  \end{pmatrix}\in M_4(\mathbb Z), $  we have $L^4P^4=(LP)^4,L^4Q^4=(LQ)^4,L^4R^4=(LR)^4.$ 
  On multiplying $L^4$ to equation $P^4+Q^4=2R^4$ ,we get $(LP)^4+(LQ)^4=2(LR)^4$.Hence $X=\begin{pmatrix}
      0&t&0&0\\t&0&0&0\\0&0&0&t\\0&0&t&0
  \end{pmatrix},Y=\begin{pmatrix}
      0&0&t&0\\0&0&0&t\\t&0&0&0\\0&t&0&o
  \end{pmatrix},Z=\begin{pmatrix}
      0&0&0&t\\0&0&t&0\\0&t&0&0\\t&0&0&0
  \end{pmatrix}$ satisfy $X^4+Y^4=2Z^4,$ for $t\in \mathbb Z.$    
\end{proof}

\begin{theorem}
    The equatiom $X^4+Y^4=2Z^4$ has infinitely many  solutions in $M_4(\mathbb Z).$
\end{theorem}

\begin{proof}
 Let $\zeta,\eta \in \mathbb Z \setminus{\{0}\}$ with $2\zeta^4-\eta^4 \neq 0,1.$ Again let $C_{\zeta,\eta}=\begin{pmatrix}
    0&1&0&0\\0&0&1&0\\0&0&0&1\\2\zeta^4-\eta^4&0&0&0
\end{pmatrix}$ ,we get $C_{\zeta,\eta}^4=(2\zeta^4-\eta^4)I_4$ implies $C_{\zeta,\eta}^4=2\zeta^4I_4-\eta^4I_4=2\begin{pmatrix}
    0&0&0&\zeta\\\zeta&0&0&0\\0&\zeta&0&0\\0&0&\zeta&0
\end{pmatrix}^4-\begin{pmatrix}
    0&0&\eta&0\\0&0&0&\eta\\\eta&0&0&0\\0&\eta&0&0
\end{pmatrix}^4.$ Hence $(C_{\zeta,\eta},D_\eta,E_\zeta)$ is the family of solutions for the Diophantine equation, where $D_\eta=\begin{pmatrix}
    0&0&\eta&0\\0&0&0&\eta\\\eta&0&0&0\\0&\eta&0&0
\end{pmatrix},E_\zeta=\begin{pmatrix}
    0&0&0&\zeta\\\zeta&0&0&0\\0&\zeta&0&0\\0&0&\zeta&0
\end{pmatrix}.$

\end{proof}

\begin{theorem}
 Let $l(\ne2)\in \mathbb Q  \setminus {\{0}\}$  be any non perfect square and 4th power free number.The field extension $\mathbb Q(\sqrt[4]l)$  has degree 4 over $\mathbb Q.$ Then the diophantine equation $X^4+Y^4=2Z^4$ has infinite number of  solutions in $M_4(\mathbb Q(\sqrt[4]l))$ for each such l.
\end{theorem}

\begin{proof}
Let $C_l=\begin{pmatrix}
    0&1&0&0\\0&0&1&0\\0&0&0&1\\2-l&0&0&0
\end{pmatrix}$ and assume $L_l=\begin{pmatrix}
    0&0&0&\sqrt[4]l\\\sqrt[4]l&0&0&0\\0&\sqrt[4]l&0&0\\0&0&\sqrt[4]l&0
\end{pmatrix}.$ Now $C_d^4=(2-d)I_4\implies C_4^d=2I_4-dI_4.$ Since $L_d^4=dI_4,$ we get $C_l^4+L_l^4=2I_4.$ Hence, $(C_l,L_l,I)$ is the family of  solutions for  each such l defined above.
   
\end{proof}

\begin{theorem}
 Let $l_1,l_2\in \mathbb Q  \setminus {\{0}\}$  be any non perfect square and 4th power free number with $l_1\ne 2l_2$.The field extension $\mathbb Q(\sqrt[4]l_1, \sqrt[4]l_2)$   over $\mathbb Q $ with $\sqrt[4]l_2 \notin\mathbb Q(\sqrt[4]l_1)$ has degree 16. Then the Diophantine equation $X^4+Y^4=2Z^4$ has infinite number of  matrix solutions in $M_4(\mathbb Q(\sqrt[4]l_1 , \sqrt[4]l_2 ))$ for each such  $l_1,l_2.$
 
\end{theorem}

\begin{proof}
Let $C_{l_1,l_2}=\begin{pmatrix}
    0&1&0&0\\0&0&1&0\\0&0&0&1\\2l_2-l_1&0&0&0
\end{pmatrix} \in M_4(\mathbb Q(\sqrt[4]l_1 , \sqrt[4]l_2 )).$ Then, we get  $C_{l_1,l_2}^4=(2l_2-l_1)I_4.$ That implies $C_{l_1,l_2}^4=2l_2I_4-l_1I_4.$ Define $N_{l_2}=\begin{pmatrix}
    0&0&0&\sqrt[4]l_2\\\sqrt[4]l_2&0&0&0\\0&\sqrt[4]l_2&0&0\\0&0&\sqrt[4]l_2&0
\end{pmatrix},M_{l_1}=\begin{pmatrix}
    0&0&\sqrt[4]l_1&0\\0&0&0&\sqrt[4]l_1\\\sqrt[4]l_1&0&0&0\\0&\sqrt[4]l_1&0&0
\end{pmatrix}.$ As $2N_{l_2}^4-M_{l_1}^4=2l_2I_4-l_1I_4,$ Hence, we have $C_{l_1,l_2}^4=2N_{l_2}^4-M_{l_1}^4,$ This implies $C_{l_1,l_2}^4+M_{l_1}^4=2N_{l_2}^4.$ Thus we have $(C_{l_1,l_2},M_{l_1},N_{l_2})$ is the family of solutions for each such $l_1,l_2.$
\end{proof}
\section{Conclusion}
This article establishes that the equation $X^4+Y^4=2Z^4$ yields infinitely many matrix solutions. These valid configurations span across multiple integer matrix dimensions and extend into specialized degree-4 and degree-16 rational field extensions. This broadly expands the known algebraic outlines for resolving higher-degree Fermat-type equations beyond standard scalar limitations.

\textbf{Conflicts of interest :} The authors declare that they have no conflicts of interest to disclose.

\end{document}